\documentclass{amsart}
\usepackage{amsfonts}
\usepackage{amsmath}
\usepackage{amssymb}
\usepackage{amsthm}
\usepackage{graphicx}
\usepackage{capt-of}
\usepackage{tikz}
\usetikzlibrary{patterns}
\usetikzlibrary{decorations.markings}
\usepackage{multirow}
\usepackage{array} 
\newcolumntype{C}{>{$}c<{$}}

\newtheorem{theorem}{Theorem}
\newtheorem{prop}{Proposition}
\newtheorem{lemma}{Lemma}
\newtheorem{rem}{Remark}
\newtheorem{exmp}{Example}

\begin{document}
\title{Triangulations with homogeneous zigzags}
\author{Mariusz Kwiatkowski, Mark Pankov,  Adam Tyc}
\subjclass[2000]{}
\address{Mariusz Kwiatkowski: Faculty of Mathematics and Computer Science, 
University of Warmia and Mazury, S{\l}oneczna 54, 10-710 Olsztyn, Poland}
\email{mkw@matman.uwm.edu.pl}

\address{Mark Pankov: Faculty of Mathematics and Computer Science, 
University of Warmia and Mazury, S{\l}oneczna 54, 10-710 Olsztyn, Poland}
\email{pankov@matman.uwm.edu.pl}

\address{Adam Tyc: Institute of Mathematics, Polish Academy of Science, \'Sniadeckich 8, 00-656 Warszawa, Poland}
\email{atyc@impan.pl}

\maketitle
\begin{abstract}
We investigate zigzags in triangulations of connected closed $2$-dimensional surfaces
and show that there is a one-to-one correspondence
between triangulations with homogeneous zigzags and closed $2$-cell embeddings of directed Eulerian graphs in surfaces.
A triangulation is called $z$-knotted if it has a single zigzag.
We construct a family of tree structured $z$-knotted spherical triangulations whose zigzags are homogeneous.
\end{abstract}

\section{Introduction}
Petrie polygons are well-known objects described by Coxeter \cite{Coxeter} (see also \cite{McMSch}).
These are skew polygons in regular polyhedra such that  any two consecutive edges, but not three, are on the same face.
Analogs of Petrie polygons for graphs embedded in surfaces are called {\it zigzags} \cite{DDS-book,Lins1} or {\it closed left-right paths} \cite{GR-book,Shank}.
Zigzags have many applications, for example, they are successfully exploited to enumerate all combinatorial possibilities for fullerenes \cite{BD}.
The case when a map, i.e an embedding of a graph in a surface, has a single zigzag is very important \cite{DDS-book,GR-book}.
Following \cite{DDS-book} we call such maps $z$-{\it knotted}.
They have nice homological properties and are closely connected to Gauss code problem 
\cite{CrRos, GR-book,Lins2}.

The studying of zigzags in $3$-regular plane graphs, in particular fullerenes, is one of the main directions of \cite{DDS-book}.
A large class of $z$-knotted $3$-regular plane graphs is obtained by using computer.
The dual objects, i.e. spherical triangulations, have the same zigzag structure. 
Zigzags in triangulations of surfaces (not necessarily orientable) are investigated in \cite{PT1,PT2,PT3}.
By \cite{PT2}, every such triangulation admits a $z$-knotted shredding.
In this paper, we describe a class of $z$-knotted triangulations which cannot be obtained by shredding.

A $z$-{\it orientation} of a map is a minimal collection of zigzags which double covers the set of edges \cite{DDS-book}.
In the $z$-knotted case, this collection contains only one zigzag and is unique up to reversing.
For every $z$-orientation we have the following two types of edges:
an edge is of type I if the distinguished zigzags pass through this edge in different directions
and an edge is of type II if they pass through the edge in the same direction. 
It is not difficult to prove that for every face in a triangulation with fixed $z$-orientation one of the following possibilities is realized:
the face contains precisely two edges of type I and the third edge is of type II or  all edges are of type II.
In the case when all faces are of the first type,
the number of edges of type I is the twofold number of edges of type II
and we say that a zigzag is {\it homogeneous} if it contains precisely two edges of type I after each edge of type II.
We describe a one-to-one correspondence between triangulations with homogeneous zigzags
and embeddings of directed Eulerian graphs in surfaces such that all edges of every face form a directed cycle (Theorem 1).
Note that directed Eulerian spherical embeddings are known also as  {\it plane alternating dimaps};
they are investigated, for example, in \cite{BHS, Farr, McCourt}.
Directed Eulerian embeddings in arbitrary surfaces are considered in \cite{BCMMcK, CGH}.

In the second part of the paper, we construct a class of $z$-knotted spherical triangulations whose zigzags are homogeneous. 
These triangulations are tree structured and we show how they can be obtained from rooted trees.

\section{Zigzags and $z$-orientations of triangulations of surfaces}
Let $M$ be a connected closed $2$-dimensional surface (not necessarily orientable).
A {\it triangulation} of $M$ is a $2$-cell embedding of a connected simple finite graph in $M$ such that all faces are triangles \cite[Section 3.1]{MT-book}.
Then the following assertions are fulfilled:
(1) every edge is contained in precisely  two distinct faces,
(2) the intersection of two distinct faces is an edge or a vertex or empty.

Let $\Gamma$ be a triangulation of $M$.
A {\it zigzag}  in $\Gamma$ is a sequence of edges $\{e_{i}\}_{i\in {\mathbb N}}$ satisfying the following conditions for every natural $i$: 
\begin{enumerate}
\item[$\bullet$] $e_{i}$ and $e_{i+1}$ are distinct edges of a certain face 
(then they have a common vertex, since every face is a triangle),
\item[$\bullet$] the faces containing $e_{i},e_{i+1}$ and $e_{i+1},e_{i+2}$ are distinct
and the edges $e_{i}$ and $e_{i+2}$ are non-intersecting.
\end{enumerate} 
Since $\Gamma$ is finite, 
there is a natural number $n>0$ such that $e_{i+n}=e_{i}$ for every natural $i$. 
In what follows, every zigzag will be presented as a cyclic sequence $e_{1},\dots,e_{n}$,
where $n$ is the smallest number satisfying the above condition.

Every zigzag is completely determined by any pair of consecutive edges  belonging to this zigzag
and for any distinct edges $e$ and $e'$ on a face there is a unique zigzag containing the sequence $e,e'$.
If $Z=\{e_{1},\dots,e_{n}\}$ is a zigzag, then the reversed sequence $Z^{-1}=\{e_{n},\dots,e_{1}\}$ also is a zigzag.
A zigzag cannot contain a sequence $e,e',\dots,e',e$ which implies that
$Z\ne Z^{-1}$ for any zigzag $Z$, i.e. a zigzag cannot be self-reversed (see, for example, \cite{PT2}).
We say that $\Gamma$ is $z$-{\it knotted} if it contains precisely two zigzags $Z$ and $Z^{-1}$, in other words,
there is a single zigzag up to reversing.

\begin{exmp}{\rm
See Fig.1 for zigzags in the three Platonic solids which are triangulations of the sphere ${\mathbb S}^2$
(the zigzags are drawn by the bold line). 
}\end{exmp}
\begin{center}
\begin{tikzpicture}[scale=0.5]
\draw[xshift=4.375cm, fill=black] (0:1.75cm) circle (3pt);
\draw[xshift=4.375cm, fill=black] (90:1.75cm) circle (3pt);
\draw[xshift=4.375cm, fill=black] (180:1.75cm) circle (3pt);
\draw[xshift=4.375cm, fill=black] (270:1.75cm) circle (3pt);
    \draw[xshift=4.375cm,thick,line width=2pt] (0:1.75cm) \foreach \x in {90,180,...,359} {
            -- (\x:1.75cm) 
        } -- cycle (90:1.75cm);
    \draw[xshift=4.375cm, dashed] (0:1.75cm) \foreach \x in {90,270} {-- (\x:1.75cm)};
    \draw[xshift=4.375cm] (0:1.75cm) \foreach \x in {0,180} {-- (\x:1.75cm)};

\draw[xshift=8.75cm, fill=black] (0:1.75cm) circle (3pt);
\draw[xshift=8.75cm, fill=black] (60:1.75cm) circle (3pt);
\draw[xshift=8.75cm, fill=black] (120:1.75cm) circle (3pt);
\draw[xshift=8.75cm, fill=black] (180:1.75cm) circle (3pt);
\draw[xshift=8.75cm, fill=black] (240:1.75cm) circle (3pt);
\draw[xshift=8.75cm, fill=black] (300:1.75cm) circle (3pt);
    \draw[xshift=8.75cm,thick,line width=2pt] (0:1.75cm) \foreach \x in {60, 120,...,359} {
            -- (\x:1.75cm) 
        } -- cycle (60:1.75cm);
    \draw[xshift=8.75cm, dashed] (0:1.75cm)--(120:1.75cm)--(240:1.75cm)--cycle;
    \draw[xshift=8.75cm] (60:1.75cm)--(180:1.75cm)--(300:1.75cm)--cycle;

\draw[xshift=13.125cm, fill=black] (0:1.75cm) circle (3pt);
\draw[xshift=13.125cm, fill=black] (36:1.75cm) circle (3pt);
\draw[xshift=13.125cm, fill=black] (72:1.75cm) circle (3pt);
\draw[xshift=13.125cm, fill=black] (108:1.75cm) circle (3pt);
\draw[xshift=13.125cm, fill=black] (144:1.75cm) circle (3pt);
\draw[xshift=13.125cm, fill=black] (180:1.75cm) circle (3pt);
\draw[xshift=13.125cm, fill=black] (216:1.75cm) circle (3pt);
\draw[xshift=13.125cm, fill=black] (252:1.75cm) circle (3pt);
\draw[xshift=13.125cm, fill=black] (288:1.75cm) circle (3pt);
\draw[xshift=13.125cm, fill=black] (324:1.75cm) circle (3pt);
    \draw[xshift=13.125cm,thick,line width=2pt] (0:1.75cm) \foreach \x in {36, 72,...,359} {
            -- (\x:1.75cm) 
        } -- cycle (36:1.75cm);
\draw[xshift=13.125cm, fill=black] (0,0) circle (3pt);
    \draw[xshift=13.125cm] (0:1.75cm)--(0,0) (72:1.75cm)--(0,0) (144:1.75cm)--(0,0) (216:1.75cm)--(0,0) (288:1.75cm)--(0,0);
    \draw[xshift=13.125cm] (0:1.75cm)--(72:1.75cm)--(144:1.75cm)--(216:1.75cm)--(288:1.75cm)--cycle;
    \draw[xshift=13.125cm, dashed] (36:1.75cm)--(0,0) (108:1.75cm)--(0,0) (180:1.75cm)--(0,0) (252:1.75cm)--(0,0) (324:1.75cm)--(0,0);
    \draw[xshift=13.125cm, dashed] (36:1.75cm)--(108:1.75cm)--(180:1.75cm)--(252:1.75cm)--(324:1.75cm)--cycle;
\end{tikzpicture}
\captionof{figure}{ }
\end{center}
\begin{exmp}{\rm
Let $BP_n$ be  the $n$-gonal bipyramid, where $1,\dots,n$ denote the consecutive vertices of the base 
and the remaining two vertices are denoted by $a,b$ (see Fig.2 for $n=3$).  
\begin{center}
\begin{tikzpicture}[scale=0.3]


\coordinate (A) at (0.5,4);
\coordinate (B) at (0.5,-5);
\coordinate (1) at (0,-1);
\coordinate (2) at (-2,0);
\coordinate (3) at (3,0);

\draw[fill=black] (A) circle (3.5pt);
\draw[fill=black] (B) circle (3.5pt);

\draw[fill=black] (1) circle (3.5pt);
\draw[fill=black] (2) circle (3.5pt);
\draw[fill=black] (3) circle (3.5pt);

\draw[thick] (3)--(1)--(2);
\draw[thick, dashed] (2)--(3);

\draw[thick] (A)--(1);
\draw[thick] (A)--(2);
\draw[thick] (A)--(3);

\draw[thick] (B)--(1);
\draw[thick] (B)--(2);
\draw[thick] (B)--(3);

\node at (3.6,0) {$3$};
\node at (-0.45,-1.525) {$2$};
\node at (-2.5,0) {$1$};

\node at (0.5,4.6) {$a$};
\node at (0.5,-5.8) {$b$};

\node[color=white] at (4,0) {$.$};

\end{tikzpicture}
\captionof{figure}{ }
\end{center}
(a). In the case when $n=2k+1$, the bipyramid $BP_n$ is $z$-knotted.
If $k$ is odd, then the unique (up to reversing) zigzag is 
$$a1,12,2b,b3,\dots,a(n-2),(n-2)(n-1),(n-1)b,bn,n1,$$
$$1a,a2,23,3b,\dots, a(n-1),(n-1)n,nb,$$
$$b1,12,2a,a3,\dots,b(n-2),(n-2)(n-1),(n-1)a,an,n1,$$
$$1b,b2,23,3a,\dots,b(n-1),(n-1)n, na.$$
If $k$ is even, then this zigzag is
$$a1,12,2b,b3,\dots,b(n-2),(n-2)(n-1), (n-1)a,an,n1,$$
$$1b,b2,23,3a,\dots,a(n-1),(n-1)n,nb,$$
$$b1,12,2a,a3,\dots,a(n-2),(n-2)(n-1),(n-1)b,bn,n1,$$
$$1a,a2,23,3b,\dots,b(n-1),(n-1)n, na.$$
(b). If $n=2k$ and $k$ is odd, 
then the bipyramid contains precisely two zigzags (up to reversing):
$$a1,12,2b,b3,34,\dots,a(n-1),(n-1)n, nb,$$
$$b1,12,2a,a3,34,\dots,b(n-1),(n-1)n, na$$
and 
$$a2,23,3b,b4,45,\dots,an,n1,1b,$$
$$b2,23,3a,a4,45,\dots,bn,n1,1a.$$
(c). In the case when $n=2k$ and $k$ is even, 
there are precisely four zigzags (up to reversing):
$$a1,12, 2b,\dots,b(n-1),(n-1)n, na;$$ 
$$b1,12, 2a,\dots,a(n-1),(n-1)n, nb;$$
$$a2,23,3b,\dots,bn,n1,1a;$$
$$b2,23,3a,\dots,an,n1,1b.$$
}\end{exmp}
See \cite{PT1,PT2} for more examples of $z$-knotted triangulations. 
Examples of $z$-knotted fullerenes can be found in \cite{DDS-book}.

Suppose that $\Gamma$ contains precisely distinct $k$ zigzags up to reversing. 
A $z$-{\it orien\-tation} of $\Gamma$ is a collection $\tau$ consisting of $k$ distinct  zigzags  such that 
for each zigzag $Z$ we have $Z\in \tau$ or $Z^{-1}\in \tau$. 
There are precisely $2^k$ distinct $z$-orientations of $\Gamma$.
For every $z$-orientation $\tau=\{Z_{1},\dots,Z_{k}\}$ the $z$-orientation $\tau^{-1}=\{Z^{-1}_{1},\dots, Z^{-1}_{k}\}$ will be called {\it reversed} to $\tau$.

Let $\tau$ be a $z$-orientation of $\Gamma$. 
For every edge $e$ of $\Gamma$ one of the following possibilities is realized:
\begin{enumerate}
\item[$\bullet$] there is a zigzag $Z\in \tau$ such that $e$ occurs in this zigzag twice and other zigzags from $\tau$
do not contain $e$,
\item[$\bullet$] there are two distinct zigzags $Z,Z'\in\tau$ such that $e$ occurs in each of these zigzags only ones 
and other zigzags from $\tau$ do not contain $e$. 
\end{enumerate}
In the first case, we say that $e$ is an {\it edge of type} I if $Z$ passes through $e$ twice in different directions; 
otherwise, $e$ is said to be an {\it edge of type} II.
Similarly, in the second case:
$e$ is an {\it edge of type} I if $Z$ and $Z'$ pass through $e$ in different directions
or $e$ is an {\it edge of type} II if $Z$ and $Z'$ pass through $e$ in the same direction.
In what follows, edges of type II will be considered together with the direction defined by $\tau$.
A vertex of $\Gamma$ is called {\it of type} I if it belongs only to edges of type I;
otherwise, we say that this is a {\it vertex of type} II.

The following statements hold for any $z$-orientation $\tau$ of $\Gamma$.

\begin{lemma} \label{lemma1}
For each vertex of type {\rm II} the number of edges of type {\rm II} which enter this vertex
is equal to the number of edges of type {\rm II} which leave it.
\end{lemma}

\begin{proof}
The number of times that the zigzags from $\tau$ enter a vertex  is equal to the number of times that these zigzags leave this vertex.
\end{proof}

\begin{prop}\label{prop-or}
For every face of $\Gamma$ one of the following possibilities is realized:
\begin{enumerate}
\item[{\rm (I)}] the face contains two edges of type {\rm I} and the third edge is of type {\rm II}, see {\rm Fig.3(a)};
\item[{\rm (II)}] all edges of the face are of type {\rm II} and form a directed cycle, see {\rm Fig.3(b)}.
\end{enumerate}
\end{prop}
A face in a triangulation is said to be {\it of type} I or {\it of type} II if the corresponding possibility is realized.
\begin{center}
\begin{tikzpicture}[scale=0.6]

\draw[fill=black] (0,2) circle (3pt);
\draw[fill=black] (-1.7320508076,-1) circle (3pt);
\draw[fill=black] (1.7320508076,-1) circle (3pt);

\draw [thick, decoration={markings,
mark=at position 0.62 with {\arrow[scale=1.5,>=stealth]{>>}}},
postaction={decorate}] (0,2) -- (-1.7320508076,-1);

\draw [thick, decoration={markings,
mark=at position 0.62 with {\arrow[scale=1.5,>=stealth]{>>}}},
postaction={decorate}] (-1.7320508076,-1) -- (1.7320508076,-1);

\draw [thick, decoration={markings,
mark=at position 0.62 with {\arrow[scale=1.5,>=stealth]{<<}}},
postaction={decorate}] (0,2) -- (1.7320508076,-1);

\node at (0,-1.65) {(a)};

\draw[fill=black] (5.1961524228,2) circle (3pt);
\draw[fill=black] (3.4641016152,-1) circle (3pt);
\draw[fill=black] (6.9282032304,-1) circle (3pt);

\draw [thick, decoration={markings,
mark=at position 0.62 with {\arrow[scale=1.5,>=stealth]{><}}},
postaction={decorate}] (5.1961524228,2) -- (3.4641016152,-1);

\draw [thick, decoration={markings,
mark=at position 0.62 with {\arrow[scale=1.5,>=stealth]{>>}}},
postaction={decorate}] (3.4641016152,-1) -- (6.9282032304,-1);

\draw [thick, decoration={markings,
mark=at position 0.62 with {\arrow[scale=1.5,>=stealth]{><}}},
postaction={decorate}] (5.1961524228,2) -- (6.9282032304,-1);

\node at (5.1961524228,-1.65) {(b)};
\end{tikzpicture}
\captionof{figure}{ }
\end{center}

\begin{proof}[Proof of Proposition \ref{prop-or}]
Consider a face whose edges are denoted by $e_{1},e_{2},e_{3}$.
Without loss of generality we can assume that the zigzag containing the sequence $e_{1},e_{2}$ belongs to $\tau$.
Let $Z$ and $Z'$ be the zigzags containing the sequences $e_{2},e_{3}$ and $e_{3},e_{1}$, respectively. 
Then $Z\in \tau$ or $Z^{-1}\in \tau$ and $Z'\in \tau$ or $Z'^{-1}\in \tau$.
An easy verification shows that for each of these four cases we obtain (I) or (II).
\end{proof}

\begin{exmp}{\rm
If $n$ is odd, then the bipyramid $BP_n$ has the unique $z$-orientation (up to reversing), see Example 2(a).
The edges $ai$ and $bi$, $i\in \{1,\dots,n\}$ are of type I and the edges on the base of the bipyramid  are of type II.
The vertices $a,b$ are of type I and the vertices on the base are of type II. All faces are of type I.
The same happens for the case when $n=2k$ and $k$ is odd if the $z$-orientation is defined by 
the two zigzags presented in  Example 2(b); 
however, all faces are of type II  if we replace one of these zigzags by the reversed.
}\end{exmp}

\begin{exmp}{\rm
Suppose that $n=2k$ and $k$ is even.
Let $Z_{1},Z_{2},Z_{3},Z_{4}$ be the zigzags from Example 2(c).
For the $z$-orientation defined by these zigzags  all faces are of type I.
If the $z$-orientation is defined by $Z_{1},Z_{2}$ and $Z^{-1}_{3}, Z^{-1}_{4}$,
then all faces are of type II. 
In the case when the $z$-orientation is defined by $Z_{1},Z_{2},Z_{3}$ and $Z^{-1}_{4}$,
there exist faces of the both types.
}\end{exmp}

\begin{rem}{\rm
If we replace a $z$-orientation by the reversed $z$-orientation,
then the type of every edge does not change (but all edges of type II reverse the directions), consequently, 
the types of vertices and faces also do not change.
For $z$-knotted triangulations we can say about the types of edges, vertices and faces without attaching to a $z$-orientation \cite{PT1}.
}\end{rem}

\section{Homogeneous zigzags in triangulations with faces of type I}
In this section, we will always suppose that $\Gamma$ is a triangulation with fixed $z$-orientation $\tau$ such that all faces in $\Gamma$ are of type I,
i.e. each face contains precisely two edges of type I and the third edge is of type II. 
If $m$ is the number of faces, then there are precisely $m$ edges of type I and $m/2$ edges of type II.
In other words, the number of edges of type I is the twofold number of edges of type II.
We say that a zigzag of $\Gamma$ is {\it homogeneous} if it is a cyclic sequence $\{e_{i},e'_{i},e''_{i}\}^{n}_{i=1}$,
where each $e_{i}$ is an edge of type II and all $e'_{i},e''_{i}$ are edges of type I.
If a zigzag is homogeneous, then the reversed zigzag also is homogeneous.  

\begin{exmp}{\rm
The zigzags of $BP_n$ are homogeneous if $n$ is odd (the $z$-knotted case) or
$n$ is even and the $z$-orientation is defined by the two zigzags from Example 2(b) or by the four zigzags from Example 2(c).
}\end{exmp}

Let $\Gamma'$ be a closed $2$-cell embedding of a connected finite simple graph in the surface $M$.
Then all faces of $\Gamma'$ are homeomorphic to a closed $2$-dimensional disc. 
For each face $F$ we take a point $v_{F}$ belonging to the interior of $F$.
We add all $v_{F}$ to the vertex set of $\Gamma'$ and connect each $v_{F}$ with every vertex of $F$ by an edge.
We obtain a triangulation of $M$ which will be denoted by ${\rm T}(\Gamma')$.

The assumption that our $2$-cell embedding is closed cannot be omitted.
Indeed, if a certain face of $\Gamma'$ is not homeomorphic to a closed $2$-dimensional disc, 
then ${\rm T}(\Gamma')$ is not a triangulation.

\begin{exmp}{\rm
If $\Gamma'$ is a triangulation of $M$,
then every zigzag $e_{1},e_{2},e_{3},\dots$ in $\Gamma'$ can be extended to a zigzag
$$e_{1},e'_{1},e''_{1},e_{2},e'_{2},e''_{2},e_{3},\dots$$
in ${\rm T}(\Gamma')$  which passes through edges of $\Gamma'$ in the opposite directions, see Fig.4.
Therefore, every $z$-orientation $\tau'$ of $\Gamma'$ induces a $z$-orientation $t(\tau')$ of ${\rm T}(\Gamma')$.
It is clear that $t(\tau'^{-1})=t(\tau')^{-1}$ and ${\rm T}(\Gamma')$ is $z$-knotted if and only if $\Gamma'$ is $z$-knotted.
}\end{exmp}
\begin{center}
\begin{tikzpicture}[scale=1.8]

\draw[fill=black] (0,0) circle (1.5pt);

\draw[fill=black] (0:2cm) circle (1.5pt);
\draw[fill=black] (-60:2cm) circle (1.5pt);
\draw[fill=black] (-120:2cm) circle (1.5pt);
\draw[fill=black] (-180:2cm) circle (1.5pt);

\draw[thick, line width=2pt] (-180:2cm) -- (0:2cm) -- (-60:2cm) -- (-120:2cm) -- (-180:2cm);
\draw[thick, line width=2pt] (-60:2cm) -- (0,0) -- (-120:2cm);

\draw [thick, decoration={markings,
mark=at position 0.59 with {\arrow[scale=2]{>}},
mark=at position 0.51 with {\arrow[scale=2,>=stealth]{<}}},
postaction={decorate}] (0:2cm) -- (-60:2cm);

\draw [thick, decoration={markings,
mark=at position 0.51 with {\arrow[scale=2,>=stealth]{<}},
mark=at position 0.59 with {\arrow[scale=2]{>}}},
postaction={decorate}] (-60:2cm) -- (0,0);

\draw [thick, decoration={markings,
mark=at position 0.51 with {\arrow[scale=2,>=stealth]{<}},
mark=at position 0.59 with {\arrow[scale=2]{>}}},
postaction={decorate}] (0,0) -- (-120:2cm);

\draw [thick] (-120:2cm) -- (-180:2cm);

\draw [thick, line width=1pt, dashed, decoration={markings,
mark=at position 0.5 with {\arrow[scale=2]{<}}},
postaction={decorate}] (-30:1.1547cm) -- (0,0);

\draw [thick, line width=1pt, dashed, decoration={markings,
mark=at position 0.55 with {\arrow[scale=2]{>}}},
postaction={decorate}] (-30:1.1547cm) -- (0:2cm);

\draw[thick, line width=1pt, dashed] (-30:1.1547cm) -- (-60:2cm);

\draw[thick, line width=1pt, dashed] (-90:1.1547cm) -- (0,0);

\draw [thick, line width=1pt, dashed, decoration={markings,
mark=at position 0.5 with {\arrow[scale=2]{<}}},
postaction={decorate}] (-90:1.1547cm) -- (-120:2cm);

\draw [thick, line width=1pt, dashed, decoration={markings,
mark=at position 0.55 with {\arrow[scale=2]{>}}},
postaction={decorate}] (-90:1.1547cm) -- (-60:2cm);

\draw [thick, line width=1pt, dashed, decoration={markings,
mark=at position 0.55 with {\arrow[scale=2]{>}}},
postaction={decorate}] (-150:1.1547cm) -- (0,0);

\draw[thick, line width=1pt, dashed] (-150:1.1547cm) -- (-120:2cm);

\draw [thick, line width=1pt, dashed, decoration={markings,
mark=at position 0.5 with {\arrow[scale=2]{<}}},
postaction={decorate}] (-150:1.1547cm) -- (-180:2cm);

\draw[fill=white] (-30:1.1547cm) circle (1.5pt);
\draw[fill=white] (-90:1.1547cm) circle (1.5pt);
\draw[fill=white] (-150:1.1547cm) circle (1.5pt);

\node at (-130:1cm) {$e_1$};
\node at (-50:1cm) {$e_2$};

\node at (-70:1.3cm) {$e''_1$};
\node at (-110:1.3cm) {$e'_1$};

\node at (-20:0.75cm) {$e'_2$};
\node at (-11.8:1.28cm) {$e''_2$};

\node at (-29:1.93cm) {$e_3$};

\end{tikzpicture}
\captionof{figure}{ }
\end{center}

Denote by $\Gamma_{II}$ the subgraph of $\Gamma$ formed by all vertices and edges of type II.

\begin{theorem}
If all zigzags of $\Gamma$ are homogeneous, 
then $\Gamma_{II}$ is a closed $2$-cell embedding of a simple Eulerian digraph 
such that every face is a directed cycle and $\Gamma={\rm T}(\Gamma_{II})$.
Conversely, if $\Gamma'$ is a closed $2$-cell embedding of a simple Eulerian digraph and every face is a directed cycle,
then the triangulation ${\rm T}(\Gamma')$ admits a unique $z$-orien\-tation 
such that all zigzags of ${\rm T}(\Gamma')$ are homogeneous
and $\Gamma'$ is a subgraph of ${\rm T}(\Gamma')$ formed by all vertices and edges of type II. 
\end{theorem}

\begin{proof}
(I). 
Let $v$ be a vertex of $\Gamma$. 
Consider all faces containing $v$ and take the edge on each of these faces which does not contain $v$.
All such edges form a cycle which will be denoted by $C(v)$.

Suppose that all zigzags of $\Gamma$ are homogeneous and
consider any edge $e_{1}$ of type II. 
Let $v_{1}$ and $v_{2}$ be the vertices of this edge such that $e_{1}$ is directed from $v_{1}$ to $v_{2}$.
We choose one of the  two faces containing $e_1$ and take in this face the vertex $v$ which does not belong to $e_1$.
Let $e'_{1}$ and $e''_{1}$ be the edges which contain $v$ and occur in a certain zigzag $Z\in \tau$ immediately after $e_{1}$, see Fig.5.
Denote by $e_2$ the third edge of the face containing $e'_{1}$ and $e''_{1}$.
This edge contains $v_{2}$ and another one vertex, say $v_{3}$.
Since $Z$ is homogeneous, the edges $e'_{1}$ and $e''_{1}$ are of type I, and consequently, $e_2$ is of type II.
The zigzag  which goes through $e'_1$ from $v$ to $v_{2}$ belongs to $\tau$
(this follows easily from the fact that $Z$ goes through $e'_{1}$ in the opposite direction and $e'_1$ is an edge of type I).
The latter guarantees that
the edge $e_{2}$ is directed from $v_{2}$ to $v_{3}$.
By our assumption, the edge $e_{3}$ which occurs in $Z$ immediately after  $e'_{1}$ and $e''_{1}$ is of type II.
This edge is directed from $v_{3}$ to a certain vertex $v_{4}$.
So, $e_{1},e_{2},e_{3}$ are consecutive edges of the cycle $C(v)$ and each $e_i$ is directed from $v_i$ to $v_{i+1}$.
Consider the zigzag from $\tau$ which contains the sequence $e_2, e''_1$. 
The next edge in this zigzag connects $v$ and $v_{4}$ (the zigzag goes from $v$ to $v_{4}$).
Let $e_{4}$ be the edge which occurs in the zigzag after it. 
Then $e_{4}$ is an edge of type II (by our assumption), it belongs to $C(v)$ and leaves $v_{4}$.
Recursively, we establish that $C(v)$ is a directed cycle formed by edges of type II and every edge containing $v$ is of type I,
i.e. $v$ is a vertex of type I. 
Now, we consider the other face containing $e_1$ and take the vertex $v'$ of this face which does not belong to $e_{1}$. 
Using the same arguments, we establish that $v'$ is a vertex of type I and $C(v')$ is a directed  cycle formed by edges of type II. 
\begin{center}
\begin{tikzpicture}[scale=1.4, xshift=0.3cm]
\draw[fill=black] (0,0) circle (1.5pt);
\draw[fill=black] (0:2cm) circle (1.5pt);
\draw[fill=black] (-45:2cm) circle (1.5pt);
\draw[fill=black] (-90:2cm) circle (1.5pt);
\draw[fill=black] (-135:2cm) circle (1.5pt);
\draw[fill=black] (-180:2cm) circle (1.5pt);

\draw[thick, line width=2pt] (0:2cm) -- (-45:2cm) -- (-90:2cm) -- (-135:2cm) -- (-180:2cm);

\draw [thick, decoration={markings,
mark=at position 0.49 with {\arrow[scale=2,>=stealth]{>}},
mark=at position 0.61 with {\arrow[scale=2,>=stealth]{>}}},
postaction={decorate}] (-180:2cm) -- (-135:2cm);

\draw [thick, decoration={markings,
mark=at position 0.49 with {\arrow[scale=2,>=stealth]{>}},
mark=at position 0.61 with {\arrow[scale=2,>=stealth]{>}}},
postaction={decorate}] (-135:2cm) -- (-90:2cm);

\draw [thick, decoration={markings,
mark=at position 0.49 with {\arrow[scale=2,>=stealth]{>}},
mark=at position 0.61 with {\arrow[scale=2,>=stealth]{>}}},
postaction={decorate}] (-90:2cm) -- (-45:2cm);

\draw [thick, decoration={markings,
mark=at position 0.49 with {\arrow[scale=2,>=stealth]{>}},
mark=at position 0.61 with {\arrow[scale=2,>=stealth]{>}}},
postaction={decorate}] (-45:2cm) -- (0:2cm);

\draw[thick, line width=1pt, dashed] (-180:2cm) -- (0,0) -- (0:2cm);
\draw[thick, line width=1pt, dashed] (-45:2cm) -- (0,0);
\draw[thick, line width=1pt, dashed] (-90:2cm) -- (0,0);
\draw[thick, line width=1pt, dashed] (-135:2cm) -- (0,0);

\node at (0.15cm,0.15cm) {$v$};

\node at (-45:2.25cm) {$v_4$};
\node at (-90:2.25cm) {$v_3$};
\node at (-135:2.25cm) {$v_2$};
\node at (-180:2.25cm) {$v_1$};

\node at (-157.5:2.08cm) {$e_1$};
\node at (-112.5:2.08cm) {$e_2$};
\node at (-67.5:2.08cm) {$e_3$};
\node at (-22.5:2.08cm) {$e_4$};

\node at (-145:1.1cm) {$e'_1$};
\node at (-100:1.1cm) {$e''_1$};

\node[color=white] at (0:2.25cm) {s};

\end{tikzpicture}
\captionof{figure}{ }
\end{center}
For every vertex $v$ of type I we can take a face containing $v$ and the edge of this face which does not contain $v$.
This edge is of type II (since the remaining two edges of the face are of type I).
The above arguments show that the following assertions are fulfilled:
\begin{enumerate}
\item[(1)] vertices of type {\rm I} exist and for every such vertex $v$ the cycle $C(v)$ is a directed cycle formed by edges of type {\rm II};
\item[(2)] for every edge of type {\rm II} there are precisely two vertices $v$ and $v'$ of type {\rm I} such that 
this edge is contained in the cycles $C(v)$ and $C(v')$.
\end{enumerate}
Similarly, for every edge $e$ of type I we take a face containing $e$;
this face contains an edge of type II which implies that $e$ connects a vertices of different types.  

Consider $\Gamma_{II}$. 
Observe that any two vertices of type II in $\Gamma$ can be connected by a path formed by edges of type II
which means that $\Gamma_{II}$ is connected.
It is easy to see that $\Gamma_{II}$ is a $2$-cell embedding of a simple digraph 
such that every face is the directed cycle $C(v)$ for a certain vertex $v$ of type I;
in particular, this $2$-cell embedding is closed. 
Lemma 1 implies that $\Gamma_{II}$ is an Eulerian digraph. 
The equality $\Gamma={\rm T}(\Gamma_{II})$ is obvious.

The following remark will be used to prove the second part of the theorem.
The conditions (1) and (2) guarantee that every zigzag of $\Gamma$ containing an edge of type II is homogeneous.
Recall that the number of edges of type I is the twofold number of edges of type II.
This implies that there is no zigzag containing edges of type I only
(since every edge occurs twice in a unique zigzag from $\tau$ or it occurs ones in precisely two distinct zigzags from $\tau$).
Therefore, every zigzag of $\Gamma$ is homogeneous if (1) and (2)  hold.
 
 (II). Suppose that $\Gamma'$ is a closed $2$-cell embedding of a simple Eulerian digraph such that every face is a directed cycle.
 
Let $e_1,\dots,e_n$ be the directed cycle formed by all edges of a certain face of $\Gamma'$.
For every $i\in \{1,\dots,n\}$ we define $j(i)=i+2({\rm mod}\, n)$ and
denote by $e'_i$ and $e''_{i}$ the edges containing the vertex $v_{F}$ and intersecting $e_{i}$ and $e_{j(i)}$, respectively.
Consider the zigzag of ${\rm T}(\Gamma')$ which contains the sequence $e_{i},e'_{i},e''_{i}, e_{j(i)}$. 
It passes through $e_{i}$ and $e_{j(i)}$ according to the directions of these edges;
and the same holds for every edge of $\Gamma'$ which occurs in this zigzag.
Such a zigzag exists for any pair formed by a face of $\Gamma'$ and an edge on this face. 
The collection of all such zigzags is a $z$-orientation of ${\rm T}(\Gamma')$ with the following properties: 
all edges of $\Gamma'$ are of type II and every $v_F$ is a vertex of type I. 
This implies that ${\rm T}(\Gamma')$ satisfies the conditions (1) and (2) which gives the claim.
\end{proof}

\begin{exmp}{\rm
If $BP_n$ is as in Example 5, then
only $a$ and $b$ are vertices of type I and $C(a)=C(b)$ is the directed cycle formed by the edges of the base of the bipyramid. 
Conversely, if all zigzags of $\Gamma$ are homogeneous and there are precisely two vertices of type I,
then $\Gamma$ is a bipyramid (easy verification). 
}\end{exmp}

\begin{exmp}{\rm
As in Example 6, we assume that $\Gamma'$ is a triangulation of $M$.
If $\tau'$ is a $z$-orientation of $\Gamma'$ such that all faces are of type II
(see \cite[Example 4]{PT3} for a $z$-knotted triangulation of ${\mathbb S}^2$ whose faces are of type II),
then $t(\tau')$ is the $z$-orientation of ${\rm T}(\Gamma')$ described in the second part of Theorem 1.
Conversely, if all zigzags of $\Gamma$ are homogeneous and for every vertex $v$ of type I the cycle $C(v)$ is a triangle,
then $\Gamma_{II}$ is a triangulation of $M$ and all faces in this triangulation are of type II 
for the $z$-orientation $\tau'$ satisfying $\tau=t(\tau')$
(recall that the triangulation $\Gamma$ is considered together with fixed  $z$-orientation $\tau$).
}\end{exmp}

\section{Gluing of triangulations}
In this section, we describe how two triangulations of special types can be glued together to get a 
$z$-knotted triangulation with homogeneous zigzags.

Let $\Gamma$ be a triangulation with fixed $z$-orientation.
The {\it face shadow} of a zigzag $e_{1},\dots,e_{n}$ is the cyclic sequence $F_{1},\dots,F_{n}$,
where $F_{i}$ is the face containing $e_{i}$ and $e_{i+1}$ if $i<n$ and the face $F_{n}$ contains $e_{n}$ and $e_{1}$.
We will use the following fact (whose proof is a simple verification):
if $e$ is an edge of type II which occurs in a certain zigzag of $\Gamma$ twice and $F,F'$ are the faces containing $e$, 
then the face shadow of this zigzag contains each of the sequences $F,F'$ and $F',F$ ones.

From this moment, we suppose that $\Gamma$ is a $z$-knotted triangulation of a surface $M$
such that all faces are of type I and the zigzag is homogeneous.
Let $e_{1}$ and $e_{2}$ be edges of type II in $\Gamma$ with a common vertex $v$ and such that
the zigzag of $\Gamma$ is a cyclic sequence of type
$$e_{1},\dots,e_{2},\dots,e_{1}, \dots,e_{2}, \dots;$$
in what follows, every such pair of edges will be called {\it special}.
Let also $F^{-}_{i}$ and $F^{+}_{i}$ be the faces containing $e_i$.
These four faces are mutually distinct (since each face contains only one edge of type II).
Every zigzag is a cyclic sequence and we can start from any of the four times when $e_{i}$  occurs in the zigzag.
Therefore, we can assume that the face shadow of the zigzag of $\Gamma$ is 
$$F^{-}_{1},F^{+}_{1},\dots, F^{-}_{2},F^{+}_{2},\dots, F^{+}_{1},F^{-}_{1},\dots, F^{+}_{2},F^{-}_{2},\dots$$
and the faces $F^{-}_{1},F^{-}_{2}$ are on the same side of $e_{1}\cup e_{2}$ and the faces $F^{+}_{1},F^{+}_{2}$ are on the other side,
see Fig.6(a).
Denote by $v_{i}$ the vertex of $e_{i}$ distinct from $v$.

Let $\Gamma'$ be a triangulation of a surface $M'$ which contains precisely two zigzags (up to reversing) and admits 
a $z$-orientation such that all faces are of type I and the zigzags are homogeneous. 
Consider edges $e'_{1}$ and $e'_{2}$ of type II in $\Gamma'$ with a common vertex $v'$ and such that
$e'_{1}$ occurs twice in one of the zigzags of $\Gamma'$ and $e'_{2}$ occurs twice in the other.
Denote by $F'^{-}_{i}$ and $F'^{+}_{i}$ the faces containing $e'_i$. 
These four faces are mutually distinct.
The shadows of the zigzags are cyclic sequences
$$F'^{-}_{1},F'^{+}_{1},\dots, F'^{+}_{1},F'^{-}_{1},\dots\;\mbox{ and }\;F'^{-}_{2},F'^{+}_{2},\dots, F'^{+}_{2},F'^{-}_{2},\dots;$$
as above, without loss of generality we can assume that  
$F'^{-}_{1},F'^{-}_{2}$ are on the same side of $e'_{1}\cup e'_{2}$ and the faces $F'^{+}_{1},F'^{+}_{2}$ are on the other side.
Let $v'_{i}$ be the vertex of $e'_{i}$ distinct from $v'$.

Our last assumption is the following:
\begin{enumerate}
\item[(*)]
for every $i\in \{1,2\}$ the edge $e'_{i}$ enters or leaves the vertex $v'$ if and only if the edge $e_{i}$ enters or leaves the vertex $v$, respectively.
\end{enumerate}
Now, we construct a triangulation of the connected sum $M\# M'$
by rearranging the triangulations $\Gamma,\Gamma'$ and gluing them together.

In the triangulation $\Gamma$,  we split up each $e_{i}$ in two edges $e^{+}_{i}$ and $e^{-}_{i}$ such that 
$v_{1}$ belongs to $e^{+}_{1}, e^{-}_{1}$ and $v_{2}$  belongs to $e^{+}_{2},e^{-}_{2}$.
Also, the vertex $v$ is splitted  in two vertices $v^{+}$ and $v^{-}$ such that 
$v^{\delta}$, $\delta\in \{+,-\}$ is a common vertex of the edges $e^{\delta}_{1}$ and $e^{\delta}_{2}$, see Fig.6(b).
The vertex of the face $F^{\delta}_{i}$ which does not belong to $e_{i}$ will be connected with $v^{\delta}$.
Similarly, we take a vertex of $\Gamma$ connected with $v$ by a certain edge $e$ and join it  with $v_{\delta}$ if the edge $e$ and the faces $F^{\delta}_{i}$ 
are on the same side of  $e_{1}\cup e_{2}$, see Fig.6(b).
In this way, we get a graph $\Gamma_{new}$ embedded in $M$.
The face of $\Gamma_{new}$ bounded by the edges $e^{+}_{1},e^{-}_{1},e^{+}_{2},e^{-}_{2}$ will be denoted by $F$.
We repeat the above construction for the triangulation $\Gamma'$ and obtain a graph $\Gamma'_{new}$ embedded in $M'$.
As above, every $e'_{i}$ is splitted in two edges $e'^{+}_{i},e'^{-}_{i}$
and the face bounded by $e'^{+}_{1},e'^{-}_{1},e'^{+}_{2},e'^{-}_{2}$ is denoted by $F'$. 
We remove the interiors of $F$ and $F'$ from $M$ and $M'$ (respectively)  and
glue together $e^{\delta}_{i}$ and $e'^{\delta}_{i}$ for $i\in \{1,2\}$ such that $v_{1}$ is identified with $v'_{1}$ and $v_{2}$ with $v'_{2}$.
We get a triangulation of $M\# M'$ which will be denoted by ${\rm G}(\Gamma, \Gamma')$.
\begin{center}
\begin{tikzpicture}[scale=1]

\begin{scope}[xshift=0.5cm]
\draw[fill=black] (0,0) circle (2pt);

\draw[fill=black] (2,0) circle (2pt);
\draw[fill=black] (-2,0) circle (2pt);
\draw[fill=black] (0,2) circle (2pt);
\draw[fill=black] (0,-2) circle (2pt);

\draw[fill=black] (-1,1.7) circle (2pt);
\draw[fill=black] (1,1.7) circle (2pt);
\draw[fill=black] (-1,-1.7) circle (2pt);
\draw[fill=black] (1,-1.7) circle (2pt);

\draw[thick] (0,0) -- (0,2);
\draw[thick] (0,0) -- (0,-2);
\draw[thick] (0,0) -- (2,0);
\draw[thick] (0,0) -- (-2,0);

\draw[thick] (-1,1.7) -- (-2,0);
\draw[thick] (-1,1.7) -- (0,0);
\draw[thick] (1,1.7) -- (2,0);
\draw[thick] (1,1.7) -- (0,0);
\draw[thick] (-1,-1.7) -- (-2,0);
\draw[thick] (-1,-1.7) -- (0,0);
\draw[thick] (1,-1.7) -- (2,0);
\draw[thick] (1,-1.7) -- (0,0);

\node at (0.3,0.13) {$v$};

\node at (-1,-0.2) {$e_1$};
\node at (1,-0.2) {$e_2$};

\node at (-1,0.66) {$F^-_1$};
\node at (1,0.66) {$F^-_2$};

\node at (-1,-0.7) {$F^+_1$};
\node at (1,-0.7) {$F^+_2$};
\end{scope}

\begin{scope}[xshift=-0.5cm]
\draw[fill=black] (8,0) circle (2pt);
\draw[fill=black] (4,0) circle (2pt);
\draw[fill=black] (6,2) circle (2pt);
\draw[fill=black] (6,-2) circle (2pt);

\draw[fill=black] (5,1.7) circle (2pt);
\draw[fill=black] (7,1.7) circle (2pt);
\draw[fill=black] (5,-1.7) circle (2pt);
\draw[fill=black] (7,-1.7) circle (2pt);

\draw[fill=black] (6,0.4) circle (2pt);
\draw[fill=black] (6,-0.4) circle (2pt);

\draw[thick] (6,0.4) -- (6,2);
\draw[thick] (6,-0.4) -- (6,-2);

\draw[thick] (5,1.7) -- (4,0);
\draw[thick] (5,1.7) -- (6,0.4);
\draw[thick] (7,1.7) -- (8,0);
\draw[thick] (7,1.7) -- (6,0.4);
\draw[thick] (5,-1.7) -- (4,0);
\draw[thick] (5,-1.7) -- (6,-0.4);
\draw[thick] (7,-1.7) -- (8,0);
\draw[thick] (7,-1.7) -- (6,-0.4);

\draw [xshift=4cm, thick] plot [smooth, tension=2] coordinates { (0.05,0) (2,0.4) (3.95,0)};
\draw [xshift=4cm, thick] plot [smooth, tension=2] coordinates { (0.05,0) (2,-0.4) (3.95,0)};

\node at (1,-2.4) {$(a)$};
\node at (6,-2.4) {$(b)$};

\node at (6,-0.15) {$v^+$};
\node at (6,0.2) {$v^-$};

\node at (5.2,-0.14) {$e^+_1$};
\node at (6.75,-0.14) {$e^+_2$};

\node at (4.8,0.12) {$e^-_1$};
\node at (7.2,0.12) {$e^-_2$};

\node at (5,0.8) {$F^-_1$};
\node at (7,0.8) {$F^-_2$};

\node at (5,-0.85) {$F^+_1$};
\node at (7,-0.85) {$F^+_2$};
\end{scope}

\end{tikzpicture}
\captionof{figure}{ }
\end{center}

\begin{lemma}\label{lemma-gluing}
${\rm G}(\Gamma, \Gamma')$ is a $z$-knotted triangulation, where all faces are of type {\rm I} and the zigzag is homogeneous. 
It contains precisely $m+m'$ vertices of type {\rm I}, 
where $m$ and $m'$ are the numbers of vertices of type {\rm I} in $\Gamma$ and $\Gamma'$, respectively.
\end{lemma}

We will use the following simple lemma to prove Lemma \ref{lemma-gluing}.

\begin{lemma}\label{gl-0}
If sequences $p_{1},\dots,p_{t},e$ and $e,q_{1},\dots,q_{s}$
are parts of zigzags in a certain triangulation and the edges $p_{t}$ and $q_{1}$ belong to distinct faces,
then the sequence 
$$p_{1},\dots,p_{t},e,q_{1},\dots,q_{s}$$
is a part of a zigzag in this triangulation.
\end{lemma}

\begin{proof}
Easy verification.
\end{proof}

\begin{proof}[Proof of Lemma \ref{lemma-gluing}]
The zigzag of $\Gamma$ can be considered as a cyclic sequence 
$$e_{1},A^{+}_{1},e_{2},A^{+}_{2},e_{1},A^{-}_{1},e_{2},A^{-}_{2},$$
where each $A^{\delta}_{i}$, $\delta\in \{+,-\}$ is a sequence of edges.
The face shadow of this zigzag is a cyclic sequence
$$F^{-}_{1},F^{+}_{1},\dots,F^{-}_{2},F^{+}_{2},\dots, F^{+}_{1},F^{-}_{1},\dots,F^{+}_{2},F^{-}_{2},\dots$$
and we can assume that the first edge from $A^{+}_{1}$ belongs to $F^{+}_{1}$.
Then the last edge from $A^{+}_{1}$ is contained in $F^{-}_{2}$.
The first and the last edges from $A^{+}_{2}$ belong to $F^{+}_{2}$ and $F^{+}_{1}$ (respectively).
The first edge from $A^{-}_{1}$ is contained in $F^{-}_{1}$ and the last edge is in $F^{+}_{2}$.
The first and the last edges from $A^{-}_{2}$ belong to $F^{-}_{2}$ and $F^{-}_{1}$ (respectively).

Similarly, the zigzags of $\Gamma'$ are cyclic sequences 
$$e'_{1},B^{+}_{1},e'_{1},B^{-}_{1}\;\mbox{ and }\;e'_{2},B^{+}_{2},e'_{2},B^{-}_{2},$$
where all $B^{\delta}_{i}$ are sequences of edges.
The corresponding face shadows are cyclic sequences
$$F'^{-}_{1},F'^{+}_{1},\dots,F'^{+}_{1},F'^{-}_{1},\dots\;\mbox{ and }\;F'^{-}_{2},F'^{+}_{2},\dots,F'^{+}_{2},F'^{-}_{2}\dots,$$
and we suppose that the first edge from $B^{+}_{i}$ belongs to $F'^{+}_{i}$. 
Then the last edge from $B^{+}_{i}$ also belongs $F'^{+}_{i}$.
The first and the last edges from $B^{-}_{i}$ belong to $F'^{-}_{i}$.

The above observations show the following sequences are parts of zigzags in the triangulation ${\rm G}(\Gamma, \Gamma')$:
\begin{enumerate}
\item[$\bullet$] $e^{\delta}_{1},A^{\delta}_{1},e^{-\delta}_{2}$ with $\delta\in \{+,-\}$
($-+=-$ and $--=+$),
\item[$\bullet$] $e^{\delta}_{2},A^{\delta}_{2},e^{\delta}_{1}$,
\item[$\bullet$] $e^{\delta}_{i},B^{\delta}_{i},e^{\delta}_{i}$ with $\delta\in \{+,-\}$.
\end{enumerate}
Lemma \ref{gl-0} implies that the cyclic sequence 
\begin{equation}\label{eq1}
e^{+}_{1},A^{+}_{1},e^{-}_{2},B^{-}_{2},e^{-}_{2},A^{-}_{2}, e^{-}_{1},B^{-}_{1},e^{-}_{1},A^{-}_{1},e^{+}_{2},B^{+}_{2},e^{+}_{2},A^{+}_{2},e^{+}_{1},B^{+}_{1}
\end{equation}
is a zigzag of ${\rm G}(\Gamma, \Gamma')$.

This zigzag contains all $A^{\delta}_{i}$ and $B^{\delta}_i$.
Also, each $e^{\delta}_{i}$ occurs in the zigzag  twice. 
So, our zigzag passes through every edge of ${\rm G}(\Gamma, \Gamma')$ twice
which means that this triangulation is $z$-knotted.

Since the zigzag of ${\rm G}(\Gamma, \Gamma')$ is the cyclic sequence \eqref{eq1},
every edge of type I from $\Gamma$ or $\Gamma'$ corresponds to an edge of type I in ${\rm G}(\Gamma, \Gamma')$.
Therefore, every face of ${\rm G}(\Gamma, \Gamma')$ contains two edges of type I, i.e. it is a face of type I.
The remaining edges of ${\rm G}(\Gamma, \Gamma')$ are of type II which implies that the zigzag of ${\rm G}(\Gamma, \Gamma')$
is homogeneous. 
Also, a vertex of ${\rm G}(\Gamma, \Gamma')$ is of type I if and only if it corresponds to a vertex of type I  in $\Gamma$ or $\Gamma'$. 
\end{proof}

\begin{rem}\label{rem2}{\rm
Keeping the above notations, we consider the case when $\Gamma'$ contains precisely four zigzags $Z^{+}_{1},Z^{-}_{1},Z^{+}_{2},Z^{-}_{2}$
(up to reversing)  and every $e'_{i}$ occurs in each of the zigzags $Z^{+}_{i},Z^{-}_{i}$ once. 
We suppose that 
the face shadow of $Z^{+}_{i}$ contains the sequence $F'^{-}_{i},F'^{+}_{i}$
and the face shadow of $Z^{-}_{i}$ contains the reversed sequence $F'^{+}_{i},F'^{-}_{i}$.
We remove from these zigzags the edge $e'_{i}$ and 
obtain two sequences of edges which will be denoted by $C^{+}_{i}$ and $C^{-}_{i}$, respectively.
For every $\delta\in \{+,-\}$ the first edge from $C^{\delta}_{i}$ belongs to $F'^{\delta}_{i}$
and the last edge is contained in $F'^{-\delta}_{i}$ (as above, $-+=-$ and $--=+$).
In other words, we have the sequence $e^{\delta}_{i},C^{\delta}_{i},e^{-\delta}_{i}$ instead of $e^{\delta}_{i},B^{\delta}_{i},e^{\delta}_{i}$
and the cyclic sequence 
\begin{equation}\label{eq2}
e^{+}_{1},A^{+}_{1},e^{-}_{2},C^{-}_{2},e^{+}_{2},A^{+}_{2}, e^{+}_{1},C^{+}_{1},e^{-}_{1},A^{-}_{1},e^{+}_{2},C^{+}_{2},e^{-}_{2},A^{-}_{2},e^{-}_{1},C^{-}_{1}
\end{equation}
is a zigzag of ${\rm G}(\Gamma, \Gamma')$.
It is easy to check that all statements from Lemma \ref{lemma-gluing} hold.
}\end{rem}

In that follows, we will consider the triangulation ${\rm G}(\Gamma, \Gamma')$ for the both cases when
$\Gamma'$ contains precisely two or four zigzags (up to reversing).

\section{Concordant special pairs}
Let $\Gamma,\Gamma'$ and $e_{1},e_{2},e'_{1},e'_{2}$ be as in the previous section.
The zigzag of $\Gamma$ is the cyclic sequence 
$$e_{1},A^{+}_{1},e_{2},A^{+}_{2},e_{1},A^{-}_{1},e_{2},A^{-}_{2}.$$
If $\Gamma'$ contains precisely two zigzags (up to reversing), then these zigzags are the cyclic sequences
$$e'_{1},B^{+}_{1},e'_{1},B^{-}_{1}\;\mbox{ and }\;e'_{2},B^{+}_{2},e'_{2},B^{-}_{2}.$$
In the case when $\Gamma'$ contains precisely four zigzags up to reversing (Remark \ref{rem2}),
$C^{+}_{i}$ and $C^{-}_{i}$ with $i\in\{1,2\}$ are sequences of edges obtained from these zigzags by removing  $e'_{i}$.

Recall that two distinct edges $c_{1}$ and $c_{2}$ of type II in $\Gamma$ form a {\it special pair} if
the zigzag of $\Gamma$ is a cyclic sequence of type
$$c_{1},\dots,c_{2},\dots,c_{1}, \dots,c_{2}, \dots$$
and the edges have a common vertex.

We say that special pairs $c_{1},c_{2}$ and $t_{1},t_{2}$ in $\Gamma$ are {\it concordant} 
if $c_{i}\ne t_{j}$ for any $i,j\in \{1,2\}$ and
the zigzag of $\Gamma$ is a cyclic sequence of type
$$c_{1},\dots,t_{1},\dots,c_{2},\dots, t_{2},\dots,c_{1},\dots,t_{1},\dots,c_{2},\dots, t_{2},\dots;$$
in particular, if $t_{1},t_{2}$ is a special pair concordant to the special pair $e_{1},e_{2}$, 
then $t_i$ belongs to both $A^{+}_{i}$ and $A^{-}_{i}$ for every $i\in \{1,2\}$. 
It is clear that this relation is symmetric.

\begin{lemma}\label{lemma-c1}
The following assertions are fulfilled:
\begin{enumerate}
\item[(A)] 
Every special pair $t_{1},t_{2}$ in $\Gamma$ concordant to the special pair $e_{1},e_{2}$ is a special pair in ${\rm G}(\Gamma, \Gamma')$.
\item[(B)] 
If two concordant special pairs $c_{1},c_{2}$ and $t_{1},t_{2}$ in $\Gamma$ both are concordant to the special pair $e_{1},e_{2}$, 
then they are concordant special pairs in ${\rm G}(\Gamma, \Gamma')$.
\end{enumerate}
\end{lemma}

\begin{proof}
(A).
This follows from the fact that the zigzag of ${\rm G}(\Gamma, \Gamma')$ is the cyclic sequence \eqref{eq1} or \eqref{eq2}
and every $t_i$ belongs to both $A^{+}_{i}$ and $A^{-}_{i}$.

(B). Without loss of generality we can assume that the zigzag of $\Gamma$ is a cyclic sequence of type
$$e_{1},\underbrace{\dots,c_{1},\dots,t_{1},\dots}_{A^{+}_{1}},e_{2},
\underbrace{\dots,c_{2},\dots, t_{2},\dots}_{A^{+}_{2}},e_{1},
\underbrace{\dots, c_{1},\dots,t_{1},\dots}_{A^{-}_{1}},e_{2},
\underbrace{\dots,c_{2},\dots, t_{2},\dots}_{A^{-}_{2}}$$
which gives the claim, since the zigzag of ${\rm G}(\Gamma, \Gamma')$ is the cyclic sequence \eqref{eq1} or \eqref{eq2}.
\end{proof}

\begin{lemma}\label{lemma-c2}
Let $c'_{1}$ and $c'_{2}$ be edges of type II in $\Gamma'$ with a common vertex 
and such that $c'_i$ belongs to both $B^{+}_{i},B^{-}_{i}$ or both $C^{+}_{i},C^{-}_{i}$ for every $i\in \{1,2\}$.
Then $c'_{1},c'_{2}$ is a special pair in ${\rm G}(\Gamma, \Gamma')$.
If $c_{1},c_{2}$ is a special pair in $\Gamma$ concordant to the special pair $e_{1},e_{2}$,
then $c_{1},c_{2}$ and $c'_{1},c'_{2}$ are concordant special pairs in ${\rm G}(\Gamma, \Gamma')$.
\end{lemma}

\begin{proof}
As above, we use the zigzag descriptions \eqref{eq1} or \eqref{eq2}.
\end{proof}

\section{Tree structured $z$-knotted spherical triangulations with homogeneous zigzags}

We apply the results of Sections 4 and 5 to the bipyramids with fixed $z$-orienta\-tions such that 
all faces are of type I and all zigzags are homogeneous. 
In this case, an edge of $BP_n$ is of type II if and only if it belongs to the base of the bipyramid.
We will use the following observations:
\begin{enumerate}
\item[(1)] 
If $n$ is odd, then $BP_n$ is $z$-knotted and any two consecutive edges of the base form a special pair.
Any two special pairs $e_{1},e_{2}$ and $c_{1},c_{2}$ in the bipyramid are concordant if $e_{i}\ne c_{j}$ for any $i,j\in \{1,2\}$.
\item[(2)] 
In the case when $n=2k$ and $k$ is odd, $BP_n$ contains precisely two zigzags (up to reversing). 
If $c$ and $c'$ are consecutive edges of the base, then $c$ occurs twice in one of these zigzags and 
$c'$ occurs twice in the other.   
\item[(3)] 
Suppose that $n=2k$ and $k$ is even. Then $BP_n$ contains precisely four zigzags (up to reversing). 
If $c$ and $c'$ are consecutive edges of the base, then $c$ occurs once in two distinct zigzags and 
$c'$ occurs once in the two others.
\end{enumerate}

Let us take two consecutive edges $e_{1},e_{2}$ in the base of $BP_{2n+1}$ together with two consecutive edges $e'_{1},e'_{2}$ in the base of $BP_{2k}$
and apply Lemma \ref{lemma-gluing} or Remark \ref{rem2} to the corresponding gluing $\Gamma_{1}={\rm G}(BP_{2n+1},BP_{2k})$.
By the lemmas from the previous section and the above observations, 
any two consecutive edges $c_{1},c_{2}$ in the base of $BP_{2n+1}$  and any two consecutive edges $c'_{1},c'_{2}$ in the base of $BP_{2k}$
form special pairs in $\Gamma_{1}$ if $c_{i}\ne e_{j}$ and $c'_{i}\ne e'_{j}$ for any $i,j\in \{1,2\}$.
The set of all such pairs $c_{1},c_{2}$ and $c'_{1},c'_{2}$ will be denoted by ${\mathcal C}_1$.
Lemmas \ref{lemma-c1} and \ref{lemma-c2} guarantee that any two pairs from ${\mathcal C}_1$
are concordant if there is no edge which occurs in both the pairs.

Therefore, we can repeat the gluing construction for any pair from ${\mathcal C}_1$ and any pair of consecutive edges in the base of $BP_{2m}$
and obtain the triangulation $\Gamma_{2}={\rm G}(\Gamma_{1},BP_{2m})$.
Any two consecutive edges in the base of each of the bipyramids $BP_{2n+1},BP_{2k},BP_{2m}$ form a special pair in $\Gamma_2$
if they are not exploited for the gluing.
Let ${\mathcal C}_{2}$ be the set of all such pairs. 
As above, any two pairs from ${\mathcal C}_2$ are concordant if there is no edge which occurs in both the pairs.
So, we can repeat the gluing construction again.

\begin{rem}{\rm
If $\Gamma$ and $\Gamma'$ are $z$-knotted bipyraminds,
then ${\rm G}(\Gamma, \Gamma')$ is not $z$-knotted. 
Similarly, the gluing of $BP_{2i}$ and $BP_{2j}$ is not $z$-knotted. 
}\end{rem} 

Consider a rooted tree whose vertices are labeled by natural numbers according to the following rules:
\begin{enumerate}
\item[$\bullet$] The root is labeled by $2k+1$, where $k$ is not less than the root degree.
\item[$\bullet$] If a vertex is not the root or a leaf, then it is labeled by $2k$ such that $k$ is not less than the degree of this vertex.
\item[$\bullet$] Every leaf is labeled by an arbitrary even number not less than $4$.
\end{enumerate}
Using this tree, we construct a $z$-knotted triangulation of the sphere ${\mathbb S}^2$ 
whose faces are of type I and the zigzag is homogeneous. 

Let $v_{1},\dots,v_{m}$ be the vertices adjacent to the root. 
Suppose that these vertices are labeled by $2i_{1},\dots,2i_{m}$, respectively.
Since $k\ge m$, the bipyramid $BP_{2k+1}$ contains at least $m$ mutually concordant special pairs.
Using these pairs we glue the bipyramids $BP_{2i_{1}},\dots,BP_{2i_{m}}$ to $BP_{2k+1}$.
Next, we take one of the vertices $v_{1},\dots,v_{m}$ which is not a leaf and 
repeat the above construction for all its neighbors different from the root.
Step by step, we construct a triangulation of ${\mathbb S}^2$ such that the gluing of $BP_{2j}$ to $BP_{2i}$
corresponds to the edge connecting the vertex labeled by $2i$ with the vertex labeled by $2j$.
This triangulation is $z$-knotted and its zigzag is homogeneous. 
It must be pointed out that such a triangulation is not uniquely defined by given tree.

\section{Final remarks}

It is an open problem to describe all cases when triangulations with homogeneous zigzags are $z$-knotted.

It was observed by M. Kwiatkowski that an analogue of Lemma \ref{lemma-gluing} 
holds for the gluing by $6k+1$ or $6k+5$ ($k=0,1,\dots$) pairs of edges of type II with common vertices
(the proof is similar to the proof of Lemma \ref{lemma-gluing}, but too many complicated details appear).
Using this fact, we can obtain $z$-knotted triangulations with homogeneous zigzags for some surfaces with non-zero genus.

Also, for every surface of even Euler characteristic (not necessarily orientable) we can construct a family of $z$-knotted maps 
whose faces are triangles and the zigzags are homogeneous; 
but the corresponding graphs contain double edges, i.e. these maps are not triangulations of surfaces.


\begin{thebibliography}{99}

\bibitem{BCMMcK}
Bonnington C.P. , Conder M., Morton M. and McKenna P., 
{\it Embedding digraphs on orientable surfaces}, J. Combin. Theory Ser. B 85(2002), 1--20.

\bibitem{BHS}
Bonnington C.P., Hartsfield N. and {\v S}ir\'a{\v n} J. 
{\it Obstructions to directed embeddings of Eulerian digraphs in the plane}, European J. Combin. 25(2004), 877--891.

\bibitem{BD}
Brinkmann G., Dress, A. W. M.,
{\it PentHex puzzles. A reliable and efficient top-down approach to fullerene-structure enumeration},
Adv. Appl. Math. 21(1998), 473--480.

\bibitem{CGH}
Chen Y., Gross J. L., and Hu X., 
{\it Enumeration of digraph embeddings}, European J. Combin.  36(2014), 660--678.

\bibitem{CrRos}
Crapo H., Rosenstiehl P., {\it On lacets and their manifolds}, Discrete Math. 233 (2001), 299--320.

\bibitem{Coxeter} Coxeter H.S.M., 
{\it Regular polytopes}, 
Dover Publications, New York 1973 (3rd ed).

\bibitem{DDS-book} 
Deza M., Dutour Sikiri\'c M., Shtogrin M.,
{\it Geometric Structure of Chemistry-relevant Graphs: zigzags and central circuit}, 
Springer 2015.

\bibitem{Farr}
Farr G. E., {\it Minors for alternating dimaps}, 
Q. J. Math. 69(2018), 285--320 .

\bibitem{GR-book}
Godsil C., Royle G., {\it Algebraic Graph Theory}, 
Graduate Texts in Mathematics 207, Springer 2001.

\bibitem{Lins1} 
Lins S., {\it Graph-encoded maps}, J. Combin. Theory, Ser. B 32(1982), 171--181.

\bibitem{Lins2} 
Lins S., Oliveira-Lima E., Silva V., {\it A homological solution for the Gauss code problem in arbitrary surfaces}, 
J. Combin. Theory, Ser. B 98(2008), 506--515.

\bibitem{McCourt}
McCourt T. A.,
{\it Growth rates of groups associated with face 2-coloured triangulations and directed Eulerian digraphs on the sphere},
Electron. J. Comb. 23(2016), P1.51, 23 p. 

\bibitem{McMSch}
McMullen P., Schulte E., {\it Abstract Regular Polytopes}, Cambridge University Press 2002.

\bibitem{MT-book}
Mohar B., Thomassen C., {\it Graphs on Surfaces}, The Johns Hopkins University Press 2001.

\bibitem{PT1} 
Pankov M., Tyc A., {\it Connected sums of z-knotted triangulations}, 
Euro. J. Comb., in press, doi:10.1016/j.ejc.2018.02.010.

\bibitem{PT2}
Pankov M., Tyc A., {\it $Z$-knotted triangulations of surfaces},
arXiv:1708.04296.

\bibitem{PT3} 
Pankov M., Tyc A., {\it On two types of $z$-monodromy in triangulations of surfaces},
arXiv:1801.06585.

\bibitem{Shank} Shank H., 
{\it The theory of left-right paths} in Combinatorial Mathematics III,
Lecture Notes in Mathematics 452, Springer 1975, 42--54.

\end{thebibliography}
\end{document}